\documentclass{article}
\usepackage{amsmath}
\usepackage{amssymb}
\usepackage{amsfonts}
\usepackage[square]{natbib}
\usepackage{a4wide}

\newcommand{\Gr}{Grass(p,n)}
\newcommand{\M}{\mathcal{M}}
\newcommand{\E}{\mathbb{R}^m}
\newcommand{\IAM}{I\!A\!M}
\newcommand{\AIAM}{A\!I\!A\!M}
\newtheorem{assumption}{Assumption}
\newtheorem{remark}{Remark}
\newtheorem{definition}{Definition}
\newtheorem{proposition}{Proposition}
\newenvironment{proof}{\noindent Proof:}{\hfill $\square$\\}
\newtheorem{lemma}{Lemma}

\title{Consensus optimization on manifolds}

\author{Alain Sarlette \and Rodolphe Sepulchre\thanks{Department of Electrical Engineering and Computer Science (Montefiore Institute), University of Li\`ege, Sart-Tilman Bldg. B28, B-4000 Li\`ege, Belgium ({\tt alain.sarlette@ulg.ac.be, r.sepulchre@ulg.ac.be}).}}

\begin{document}
\pagestyle{myheadings}
\markboth{SIAM/SICON preprint: subm 09/2006; rev 11/2007; acc 04/2008; publ? $\geq$ 2009.}{SIAM/SICON preprint: subm 09/2006; rev 11/2007; acc 04/2008; publ? $\geq$ 2009.}
\maketitle
\thanks{This paper presents research results of the Belgian Network DYSCO (Dynamical Systems, Control, and Optimization), funded by the Interuniversity Attraction Poles Programme, initiated by the Belgian State, Science Policy Office. The scientific responsibility rests with its authors. Alain Sarlette is supported as an FNRS fellow (Belgian Fund for Scientific Research). The authors want to thank P.-A. Absil for valuable discussions.
Preprint submitted to SIAM/SICON September 2006; revised November 2007; accepted for publication April 2008; publication date unknown before 2009.}


\begin{abstract}
The present paper considers distributed consensus algorithms that involve $N$ agents evolving on a connected compact homogeneous manifold. The agents track no external reference and communicate their relative
state according to a communication graph. The consensus problem is formulated in terms of the extrema of a cost function. This leads to efficient gradient algorithms to synchronize (i.e. maximizing the consensus) or balance (i.e. minimizing the consensus) the agents; a convenient adaptation of the gradient algorithms is used when the communication graph is directed and time-varying. The cost function is linked to a specific centroid definition on manifolds, introduced here as the \emph{induced arithmetic mean}, that is easily computable in closed form and may be of independent interest for a number of manifolds. The special orthogonal group $SO(n)$ and the Grassmann manifold $\Gr$ are treated as original examples. A link is also drawn with the many existing results on the circle.
\end{abstract}


\pagestyle{myheadings}
\thispagestyle{plain}



\section{Introduction}
The distributed computation of means/averages of datasets (in an algorithmic setting) and the \emph{synchronization} of a set of agents (in a control setting) --- i.e. driving all the agents to a common point in state space --- are ubiquitous tasks in current engineering problems. Likewise, spreading a set of agents in the available state space --- linked to the definition of \emph{balancing} in \S 4 --- is a classical problem of growing interest. Practical applications include autonomous swarm/formation operation (e.g. \cite{SPL2005, LEONARDgen, JADBABAIE, JandK, SO3book}), distributed decision making (e.g. \cite{olfati,Tsitsiklis2}), neural and communication networks (e.g. \cite{Tsitsiklis1, HOPFIELD}), clustering and other reduction methods (e.g. \cite{TheisGr}), optimal covering or coding (e.g. \cite{Codingbook, Barg1, Conway, BULLO}) and other fields where averaging/synchronizing or distributing a set of points appear as sub-problems. In a modeling framework, the understanding of synchronization or more generally swarm behavior has also led to many important studies (e.g. \cite{Kuramoto, Sync, VICSEK}).

Synchronization algorithms are well understood in Euclidean spaces (e.g. \cite{MOREAU, MOREAU2, Tsitsiklis2, olfati}). They are based on the natural definition and distributed computation of the centroid in $\mathbb{R}^m$. However, many applications above involve manifolds that are not homeomorphic to an Euclidean space. Even for formations moving in $\mathbb{R}^2$ or $\mathbb{R}^3$, the agents' orientations evolve in a manifold $SO(2) \cong S^1$ or $SO(3)$.
Balancing only makes sense on compact state spaces; though many theoretical results concern convex or star-shaped subsets of $\E$ (e.g. \cite{BULLO}), most applications involve compact manifolds. It seems that the study of global synchronization or balancing in non-Euclidean manifolds is not widely covered in the literature, except for the circle.

The present paper proposes algorithms for global synchronization and balancing --- grouped under the term \emph{consensus} --- on connected compact homogeneous manifolds. A homogeneous manifold $\M$ is isomorphic to the quotient of two Lie groups. Intuitively, it is a manifold on which ``all points are equivalent''. This makes the problem symmetric with respect to the absolute position on the manifold and allows to focus on \emph{configurations} of the swarm, i.e. \emph{relative positions} of the agents.

The main idea is to embed $\M$ in $\E$ and measure distances between agents in $\E$ in order to build a convenient cost function for an optimization-based approach. The related centroid on $\M$ may be interesting on its own account; it is therefore studied in more detail in \S 3.

Throughout the paper, the abstract concepts are illustrated on the special orthogonal group $SO(n)$, the Grassmann manifold $\Gr$ of $p$-dimensional vector spaces in $\mathbb{R}^n$ and sometimes the circle $S^1$, which is in fact isomorphic to both $SO(2)$ and $Grass(1,2)$. Other manifolds to which the present framework could be applied include the $n$-dimensional spheres $S^n$ and the connected compact Lie groups. The circle $S^1$ is the simplest example; it links the present work to existing results in \cite{SPL2005, SPLbook, MY2, TAC2}. $SO(n)$ is important in control applications as the natural state space for orientations of $n$-dimensional rigid bodies. $\Gr$ rather appears in algorithmic problems; \cite{Conway} mentions the optimal placement of $N$ laser beams for cancer treatment and the projection of multi-dimensional data on $N$ representative planes as practical applications of optimal distributions on $\Gr$.

The paper is organized as follows. Previous work is briefly reviewed in \S 1.1. Section 2 introduces concepts and notations about graph theory, $SO(n)$ and $\Gr$. Section 3 is devoted to the induced arithmetic mean. A definition of consensus is presented in \S 4. Section 5 introduces a cost function to express the consensus problem in an optimization setting. Section 6 derives gradient algorithms based on this cost function, the only communicated information being the relative positions of interconnected agents; convergence is proved for any connected, fixed and undirected communication graph. Algorithms whose convergence properties can be guaranteed under possibly directed, time-varying and disconnected communication graphs are presented in \S 7; they employ an auxiliary variable that evolves in the embedding space $\E$.

\subsection{Previous work}
Most of the work related to synchronization and balancing on manifolds concerns the circle $S^1$. The most extensive literature on the subject derives from the Kuramoto model (see \cite{StrKur} for a review). Recently however, synchronization on the circle has been considered from a control perspective, the state variables representing headings of agents in the plane. Most results cover \emph{local} convergence \cite{JADBABAIE, MOREAU}. An interesting set of \emph{globally} convergent algorithms in $SE(2) = S^1 \times \mathbb{R}^2$ is presented in \cite{SPL2005}, but they require all-to-all communication. Some problems related to global discrete-time synchronization on $S^1$ under different communication constraints are discussed in \cite{MY1}, where connections of the control problem with various existing models are made. Stronger results are presented in \cite{MY2} for global synchronization and balancing on $S^1$ with varying, directed communication links, at the cost of introducing estimator variables that communicating agents must exchange. Finally, \cite{TAC2} presents results on $SE(2)$ similar to those of \cite{SPL2005} but under relaxed communication assumptions, using among others the estimator strategy of \cite{MY2, LUCA1}.

Several authors have already presented algorithms that asymptotically synchronize satellite attitudes, involving the rotation group $SO(3)$. They often rely on tracking a common external reference (e.g. \cite{BeardOnSats}) or leader (e.g. \cite{NorwayLeaderFollower,SO3book}). The use of the convenient but non-unique quaternion representation for $SO(3)$ produces unwanted artefacts in the satellites' motions. Attitude synchronization without common references and quaternion artefacts is studied in \cite{Sujit1}; using the same distance measure as the present work, an artificial coupling potential is built to establish local stability. All these approaches explicitly incorporate the second-order rigid-body dynamics. In accordance with the consensus approach, the present paper reduces the agents to first-order kinematic models to focus on (almost) global convergence for various agent interconnections, without any leader or external reference. Application of this framework in a mechanical setting is discussed in a separate paper \cite{MY3}.
\vspace{2mm}

Synchronization or balancing on a manifold $\M$ is closely related to the definition and computation of a mean or centroid of points on $\M$, a basic problem that has attracted somewhat more attention, as can be seen from \cite{MoM:Hueper, MoM:B&S, MoM:Galp} among others.

A key element of the present paper is the computation of a centroid in the embedding space $\E$ of $\M$, which is then projected onto $\M$. This is connected to the ``projected arithmetic mean'' defined in \cite{MOAKHER_SO3} for $SO(3)$. In fact, the computation of statistics in a larger and simpler embedding manifold (usually Euclidean space) and projecting the result back onto the original manifold, goes back to 1972 \cite{1972mean}.

A short example in \cite{PAGrass} addresses the computation of a ``centroid of subspaces'', without much theoretical analysis; in fact, our algorithms on $\Gr$ are similar and can eventually be viewed as generalizing the developments in \cite{PAGrass} in the framework of consensus and synchronization. More recently, \cite{TheisGr} uses the centroid associated to the projector representation of $\Gr$, exactly as is done below but without going into theoretical details, to compute the cluster centers in a clustering algorithm. The distance measure associated to this centroid on $\Gr$ is called the \emph{chordal distance} in \cite{Conway,Barg1} where it is used to derive optimal distributions (``packings'') of $N$ agents on some specific Grassmann manifolds.\vspace{2mm}

Finally, the topic of optimization-based algorithm design on manifolds has considerably developed over the last decades (see e.g. \cite{Brockett1}, \cite{Edelman99} and the books \cite{H&Mbook, PAbook}).



\section{Preliminaries}

\subsection{Elements of graph theory} 
Consensus among a group of agents depends on the available communication links.
When considering limited agent interconnections, it is customary to represent communication links by means of a \emph{graph}. The graph $G$ is composed of $N$ \emph{vertices} (the $N$ agents) and contains the \emph{edge} $(j,k)$ if agent $j$ sends information to agent $k$, which is denoted $j \rightsquigarrow k$. A positive \emph{weight} $a_{jk}$ is associated to each edge $(j,k)$ to obtain a weighted graph; the weight is extended to any pair of vertices by imposing $a_{jk} = 0$ iff $(j,k)$ does not belong to the edges of $G$. The full notation for the resulting \emph{digraph} (directed graph) is $G(V,E,A)$ where $V=\lbrace \text{vertices} \rbrace$, $E=\lbrace \text{edges} \rbrace$ and $A$, containing the $a_{jk}$, is the \emph{adjacency matrix}. The convention $a_{kk}=0$ $\forall k$ is assumed for the representation of communication links.

The \emph{out-degree} of a vertex $k$ is the quantity $d^{(o)}_k = \sum_{j=1}^N a_{kj}$ of information leaving $k$ towards other agents; its \emph{in-degree} is the quantity $d^{(i)}_k = \sum_{j=1}^N a_{jk}$ of information received by $k$ from other agents. These degrees can be assembled in diagonal matrices $D^{(o)}$ and $D^{(i)}$. A graph is \emph{balanced} if $D^{(o)} = D^{(i)}$. This is satisfied in particular by \emph{undirected} graphs, for which $A = A^T$. A graph is \emph{bidirectional} if $(j,k) \in E \Leftrightarrow (k,j) \in E$ (but not necessarily $A = A^T$).

The \emph{Laplacian} $L$ of a graph is $L = D - A$. For directed graphs, $D^{(i)}$ or $D^{(o)}$ can be used, leading to the in-Laplacian $L^{(i)} = D^{(i)} - A$ and the out-Laplacian $L^{(o)} = D^{(o)} - A$. By construction, $L^{(i)}$ has zero column sums and $L^{(o)}$ has zero row sums. The spectrum of the Laplacian reflects several interesting properties of the associated graph, specially in the case of undirected graphs (see for example \cite{Graphs}).

$G(V,E,A)$ is \emph{strongly connected} if it contains a directed path from any vertex $j$ to any vertex $l$ (i.e. a sequence of vertices starting with $j$ and ending with $l$ such that $(v_k,v_{k+1}) \in E$ for any two consecutive vertices $v_k$, $v_{k+1}$); $G$ is \emph{weakly connected} if there is such a path in the \emph{associated undirected graph}, with adjacency matrix $A+A^T$.\vspace{2mm}

For time-varying interconnections, a time-varying graph $G(t)$ is used and all the previously defined elements simply depend on time. If the elements of $A(t)$ are bounded and satisfy some threshold $a_{jk}(t) \geq \delta > 0$ $\forall (j,k) \in E(t)$ and $\forall t$, then $G(t)$ is called a \emph{$\delta$-digraph}. The present paper always considers $\delta$-digraphs.

In a $\delta$-digraph $G(V,E,A)$, vertex $j$ is said to be \emph{connected} to vertex $k$ \emph{across $[t_1,t_2]$} if there is a path from $j$ to $k$ in the digraph
$\bar{G}(V,\bar{E},\bar{A})$ defined by 
\begin{eqnarray*}
& & \bar{a}_{jk} = \left\lbrace 
\begin{array}{ll}
\int_{t_1}^{t_2} a_{jk}(t) dt & \text{if } \int_{t_1}^{t_2} a_{jk}(t) dt \geq \delta \\
0 & \text{if } \int_{t_1}^{t_2} a_{jk}(t) dt < \delta
\end{array} \right. \\
& & (j,k) \in \bar{E} \text{ iff } \bar{a}_{jk} \neq 0 \; .
\end{eqnarray*}
$\bar{G}$ can be seen as a time-integrated graph while the $\delta$-criterion prevents vanishing edges. A $\delta$-digraph $G(t)$ is called \emph{uniformly connected} if there exist a vertex $k$ and a time horizon $T > 0$ such that $\forall t$, $k$ is connected to all other vertices across $[t,t+T]$.

\subsection{Specific manifolds} 
The concepts presented in this paper are illustrated on two particular manifolds: $SO(n)$ and $\Gr$.\vspace{2mm}

\paragraph*{The special orthogonal Lie group $SO(n)$} It can be viewed as the set of positively oriented orthonormal bases of $\mathbb{R}^n$, or equivalently as the set of rotation matrices in $\mathbb{R}^n$; it is the natural state space for the orientation of a rigid body in $\mathbb{R}^n$. In its canonical representation, used in the present paper, a point of $SO(n)$ is characterized by a real $n \times n$ orthogonal matrix $Q$ with determinant equal to $+1$. $SO(n)$ is homogeneous (as any Lie group), compact and connected. It has dimension $n(n-1)/2$.\vspace{2mm}

\paragraph*{The Grassmann manifold $\Gr$} Each point on $\Gr$ denotes a $p$-dimensional subspace $\mathcal{Y}$ of $\mathbb{R}^n$. The dimension of $\Gr$ is $p (n-p)$. Since $Grass(n-p,n)$ is isomorphic to $\Gr$ by identifying orthogonally complementary subspaces, this paper assumes without loss of generality that $p \leq \frac{n}{2}$. For the special case $p=1$, the Grassmann manifold $Grass(1,n)$ is also known as the \emph{projective space} in dimension $n$. $\Gr$ is connected, compact and homogeneous as the quotient of the orthogonal Lie group $O(n)$ by $O(p) \times O(n-p)$. Indeed, $\mathcal{Y} \in \Gr$ can be represented for instance by a (not necessarily positively oriented) orthonormal basis $Q \in O(n)$ whose first $p$ column-vectors span $\mathcal{Y}$; the same point $\mathcal{Y} \in \Gr$ is represented by any $Q$ whose first $p$ column-vectors span $\mathcal{Y}$ ($O(p)$-symmetry) and whose last $n-p$ column-vectors span the orthogonal complement of $\mathcal{Y}$ ($O(n-p)$-symmetry). Other quotient structures for $\Gr$ are discussed in \cite{PAGrass}.

A matrix manifold representation of $\Gr$ found in \cite{PAGrass} assigns to $\mathcal{Y}$ any $n \times p$ matrix $Y$ of $p$ orthonormal column-vectors spanning $\mathcal{Y}$ ($p$-basis representation); all $Y$ corresponding to rotations and reflections of the $p$ column-vectors in $\mathcal{Y}$ represent the same $\mathcal{Y}$ ($O(p)$-symmetry), so this representation is not unique. The dimension of this representation is $np-p(p+1)/2$. In \cite{SPGrass}, a point of $\Gr$ is represented by $\Pi = Y \, Y^T$, the orthonormal projector on $\mathcal{Y}$ (projector representation); using the orthonormal projector on the space orthogonal to $\mathcal{Y}$, $\Pi_{\bot} = I_n - Y \, Y^T$ where $I_n$ denotes the $n \times n$ identity matrix, is strictly equivalent. The main advantage of this representation is that there exists a bijection between $\Gr$ and the orthonormal projectors of rank $p$, such that the projector representation makes $\Gr$ an embedded submanifold of the cone $\mathbb{S}_{n}^{+}$ of $n \times n$ symmetric positive semi-definite matrices. A disadvantage of this representation is its large dimension $n(n+1)/2$.



\section{The induced arithmetic mean} A homogeneous manifold $\M$ is a manifold with a transitive group action by a Lie group $\mathcal{G}$: it is isomorphic to the quotient manifold $\mathcal{G}/\mathcal{H}$ of a group $\mathcal{G}$ by one of its subgroups $\mathcal{H}$. Informally, it can be seen as a manifold on which ``all points are equivalent''. The present paper considers connected compact homogeneous manifolds satisfying the following embedding property.\vspace{2mm}

\begin{assumption}\label{Ass1}
$\M$ is a connected compact homogeneous manifold smoothly embedded in $\E$ with the Euclidean norm $\Vert y \Vert = r_{\M}$ constant over $y \in \M$. The Lie group $\mathcal{G}$ acts as a subgroup of the orthogonal group on $\E$.
\end{assumption}
\vspace{2mm}

It is a well-known fact of differential geometry that any smooth $\tfrac{m}{2}$-dimensional Riemannian manifold can be smoothly embedded in $\mathbb{R}^{m}$. The additional condition $\Vert y \Vert = r_{\M}$ is in agreement with the fact that all points on $\M$ should be equivalent. It is sometimes preferred to represent $y \in \M$ by a matrix $B \in \mathbb{R}^{n_1 \times n_2}$ instead of a vector. Componentwise identification $\mathbb{R}^{n_1 \times n_2} \cong \mathbb{R}^m$ is assumed whenever necessary; the corresponding norm is the \emph{Frobenius norm} $\Vert B \Vert = \sqrt{\mathrm{trace}(B^T B)}$.
\vspace{2mm}

Consider a set of $N$ agents on a manifold $\M$ satisfying Assumption \ref{Ass1}. The position of agent $k$ is denoted by $y_k$ and its weight by $w_k$.\vspace{2mm}

\begin{definition}\label{Def1}
The \textbf{induced arithmetic mean} $\IAM \subseteq \M$ of $N$ agents of weights $w_k > 0$ and positions $y_k \in \M$, $k=1...N$, is the set of points in $\M$ that globally minimize the weighted sum of squared Euclidean distances in $\E$ to each $y_k$\emph{:}
\begin{equation}\label{Def:IAM}
\IAM \; = \; \mathop{\mathrm{argmin}}_{c \in \M} {\textstyle \sum_{k=1}^N} \, w_k \; d_{\E}^2(y_k,c) \; = \; \mathop{\mathrm{argmin}}_{c \in \M} {\textstyle \sum_{k=1}^N} \, w_k \; (y_k-c)^T(y_k-c) \; .
\end{equation}
The \textbf{anti-[induced arithmetic mean]} $\AIAM \subseteq \M$ is the set of points in $\M$ that globally maximize the weighted sum of squared Euclidean distances \emph{in $\E$} to each $y_k$\emph{:}
\begin{equation}\label{Def:AIAM}
\AIAM \; =  \; \mathop{\mathrm{argmax}}_{c \in \M} {\textstyle \sum_{k=1}^N} \, w_k \; d_{\E}^2(y_k,c) \; = \; \mathop{\mathrm{argmax}}_{c \in \M} {\textstyle \sum_{k=1}^N} \, w_k \; (y_k-c)^T(y_k-c) \; .
\end{equation}
\end{definition}

The terminology is derived from \cite{MOAKHER_SO3} where the $\IAM$ on $SO(3)$ is called the \emph{projected arithmetic mean}. The point in Definition \ref{Def1} is that distances are measured \emph{in the embedding space} $\E$. It thereby differs from the canonical definition of mean of $N$ agents on $\M$, the \emph{Karcher mean} \cite{Karcher,RiemannStatsPennec,KarcherMeanGroisser,MoM:Hueper}, which uses the geodesic distance $d_{\M}$ along the Riemannian manifold $\M$ (with, in the present setting, the Riemannian metric induced by the embedding of $\M$ in $\E$):
$$C_{Karcher} = \mathop{\mathrm{argmin}}_{c \in \M} \; {\textstyle \sum_{k=1}^N} \, w_k \; d_{\M}^2(y_k,c) \; .$$
The induced arithmetic mean has the following properties.
\begin{enumerate}
\item The $\IAM$ of a single point $y_1$ is the point itself.
\item The $\IAM$ is invariant under permutations of agents of equal weights.
\item The $\IAM$ commutes with the symmetry group of the homogeneous manifold.
\item The $\IAM$ does not always reduce to a single point.
\end{enumerate}
The last feature seems unavoidable for any mean (including the Karcher mean) that satisfies the other properties. The main advantage of the $\IAM$ over the Karcher mean is computational. The $\IAM$ and $\AIAM$ are closely related to the centroid.\vspace{2mm}

\begin{definition}\label{Def2} The \textbf{centroid} $C_e \in \E$ of $N$ weighted agents located on $\M$ is
$$C_e = \tfrac{1}{W} \, {\textstyle \sum_{k=1}^N} \, w_k \, y_k \; , \hspace{10mm} \text{where } \; W = {\textstyle \sum_{k=1}^N} \, w_k \; .$$
\end{definition}
Since $\Vert c \Vert = r_{\M}$ for $c \in \M$ by Assumption \ref{Ass1}, equivalent definitions for the $\IAM$ and $\AIAM$ are
\begin{equation}\label{Def:IAM&AIAM:Alt2}
\IAM = \mathop{\mathrm{argmax}}_{c \in \M} (c^T \, C_e) \hspace{5mm} \text{ and } \hspace{5mm}
\AIAM = \mathop{\mathrm{argmax}}_{c \in \M} (- c^T \, C_e) \; .
\end{equation}
Hence, computing the $\IAM$ and $\AIAM$ just involves a search for the global maximizers of a linear function on $\E$ in a very regular search space $\M$. Local maximization methods would even suffice if the linear function had no maxima on $\M$ other than the global maxima. This is the case for any linear function on $SO(n)$ and $\Gr$ (see \S 3.1) as well as the $n$-dimensional sphere $S^n$ in $\mathbb{R}^{n+1}$. Not knowing whether this property holds for all manifolds satisfying Assumption \ref{Ass1}, we formulate the following blanket assumption.\vspace{2mm}

\begin{assumption}\label{Ass2}
The local maxima of any linear function $f(c) = c^T \, b$ over $c \in \M$, with $b$ fixed in $\E$, are all global maxima.
\end{assumption}


\subsection{Examples} These examples exclusively consider the $\IAM$; from (\ref{Def:IAM&AIAM:Alt2}), the conclusions for the $\AIAM$ are simply obtained by replacing $C_e$ with $-C_e$.\vspace{2mm}

\paragraph*{The circle} The circle embedded in $\mathbb{R}^2$ with its center at the origin satisfies Assumptions \ref{Ass1} and \ref{Ass2}. The $\IAM$ is simply the central projection of $C_e$ onto the circle. Hence it corresponds to the whole circle if $C_e = 0$ and reduces to a single point in other situations. The $\IAM$ uses the chordal distance between points, while the Karcher mean would use arclength distance.\vspace{2mm}

\paragraph*{The special orthogonal group} The embedding of $SO(n)$ as orthogonal matrices $Q \in \mathbb{R}^{n \times n}$, $\mathrm{det}(Q) > 0$, satisfies Assumption \ref{Ass1} since $\Vert Q \Vert = \sqrt{\mathrm{trace}(Q^T Q)} = \sqrt{n}$. It also satisfies Assumption \ref{Ass2} (proof in \S 6). $C_e = \sum_k Q_k$ is a general $n \times n$ matrix. The $\IAM$ is linked to the \emph{polar decomposition} of $C_e$. Any matrix $B$ can be decomposed into $U R$ with $U$ orthogonal and $R$ symmetric positive semi-definite; $R$ is always unique, $U$ is unique if $B$ is non-singular \cite{RefMatrix}. Each $U$ is a global minimizer of $d_{\mathbb{R}^{n \times n}}(c,B)$ over $c \in O(n)$. Thus, if $\mathrm{det}(C_e) \geq 0$, the $\IAM$ contains all matrices $U$: $\mathrm{det}(U) > 0$ obtained from the polar decomposition of $C_e$; this was already noticed in \cite{MOAKHER_SO3}. When $\mathrm{det}(C_e) < 0$, the result is more complicated but still has a closed-form solution.\vspace{2mm}

\begin{proposition}\label{Prop1}
Consider $U$ an orthogonal matrix obtained from the polar decomposition $C_e = U R$. The $\IAM$ of $N$ points on $SO(n)$ is characterized as follows.
\begin{enumerate}
\item[] If $\mathrm{det}(C_e) \geq 0$, then $\IAM=\lbrace U: \mathrm{det}(U) > 0 \rbrace$. It reduces to a single point if the multiplicity of 0 as an eigenvalue of $C_e$ is less or equal to 1.
\item[] If $\mathrm{det}(C_e) \leq 0$, then $\IAM=\lbrace U H J H^T \rbrace$ where $\mathrm{det}(U) < 0$, $H$ contains the orthonormalized eigenvectors of $R$ with an eigenvector corresponding to the smallest eigenvalue of $R$ in the first column, and
$J=\left(
\begin{array}{cc}
-1 & 0 \\
0 & I_{n-1}
\end{array}
 \right) \; .$
The $\IAM$ reduces to a single point if the smallest eigenvalue of $R$ has multiplicity $1$.
\end{enumerate}
\end{proposition}
\vspace{2mm}

\begin{proof}
It is provided in \S 6 after introducing further necessary material to compute the critical points of $c^T C_e$, among which the local maxima are selected.
\end{proof}\vspace{2mm}

\paragraph*{The Grassmann manifold} The representation of $\Gr$ with $p$-bases $Y_k$ is not an embedding and cannot be used in the proposed framework, because the $p$-dimensional subspace of $\mathbb{R}^n$ spanned by the columns of $C_e = \sum_k Y_k$ would depend on the particular matrices $Y_k$ chosen to represent the subspaces $\mathcal{Y}_k$. The $\IAM$ is defined with the projector representation, embedding $\Gr$ in $\mathbb{S}_n^+$. The latter satisfies Assumption \ref{Ass1}; the Frobenius norm of a $p$-rank projector is $\sqrt{p}$. It also satisfies Assumption \ref{Ass2} (proof in \S 6). The centroid $C_e$ of $N$ projectors is generally a symmetric positive semi-definite matrix of rank $\geq p$.\vspace{2mm}

\begin{proposition}\label{Prop2}
The $\IAM$ contains all dominant $p$-eigenspaces of $C_e$. It reduces to a single point if the $p$-largest and $(p+1)$-largest eigenvalues of $C_e$ are different.
\end{proposition}\vspace{2mm}

\begin{proof}
Following the same lines as for $SO(n)$, it is postponed to \S 6.\end{proof}
\vspace{2mm}

In fact, for $\mathcal{Y} \in \Gr$ with a $p$-basis $Y$ and the projector $\Pi_{\mathcal{Y}}= Y Y^T$, the cost function in (\ref{Def:IAM&AIAM:Alt2}) becomes
\begin{equation}\label{GrIAM_as_Rayleigh}
f(\Pi_{\mathcal{Y}}) = \mathrm{trace}(\Pi_{\mathcal{Y}} C_e) = \mathrm{trace}(Y^T C_e Y) = \mathrm{trace}((Y^T Y)^{-1}\, Y^T C_e Y)
\end{equation}
where the last expression is equal to the generalized Rayleigh quotient for the computation of the dominant $p$-eigenspace of $C_e$. The computation of eigenspaces from cost function (\ref{GrIAM_as_Rayleigh}) is extensively covered in \cite{PAGrass, PAbook}. Furthermore, it is a well-known fact of linear algebra that the $p$ largest eigenvalues (the others being $0$) of $\Pi_{\mathcal{Y}} \Pi_k$ are the squared cosines of the principal angles $\phi_{k}^i$, $i=1...p$, between subspaces $\mathcal{Y}$ and $\mathcal{Y}_k$. This provides a geometrical meaning for the $\IAM$ of subspaces: it minimizes the sum of squared sines of principal angles between the set of subspaces $\mathcal{Y}_k$, $k=1...N$, and a centroid candidate subspace $\mathcal{Y}$, i.e. $\IAM = \mathop{\mathrm{argmin}}_{\mathcal{Y}} \, \sum_{k=1}^N \, \sum_{i=1}^p \, \sin^2( \phi_{k}^{i} )$. The Karcher mean admits the same formula with $\sin^2( \phi_{k}^{i} )$ replaced by $(\phi_{k}^{i})^2$ \cite{Conway}.



\section{Consensus} Consider a set of agents with positions $y_k$, $k=1...N$, on a manifold $\M$ satisfying Assumption \ref{Ass1}. The rest of this paper assumes equal weights $w_k=1$ $\forall k$; extension to weighted agents is straightforward. Suppose that the agents are interconnected according to a fixed digraph $G$ of adjacency matrix $A=  [a_{jk}]$.\vspace{2mm}

\begin{definition}\label{Def3}
\textbf{Synchronization} is the configuration where $y_j = y_k$ $\forall j,k$.
\end{definition}

\begin{definition}\label{Def4}
A \textbf{consensus} configuration with graph $G$ is a configuration where each agent $k$ is located at a point of the $\IAM$ of its neighbors $j \rightsquigarrow k$, weighted according to the strength of the corresponding edge. Similarly, an \textbf{anti-consensus} configuration satisfies this definition with $\IAM$ replaced by $\AIAM$.
\begin{equation}\label{Def:Consensus}
\text{\textbf{Consensus:}} \qquad y_k \in \mathop{\mathrm{argmax}}_{c \in \M} \left(c^T \; {\textstyle \sum_{j=1}^N} \, a_{jk} \, y_j \right) \qquad \forall k \; .
\end{equation}
\begin{equation}\label{Def:AntiConsensus}
\text{\textbf{Anti-consensus:}} \qquad y_k \in \mathop{\mathrm{argmin}}_{c \in \M} \left(c^T \; {\textstyle \sum_{j=1}^N} \, a_{jk} \, y_j \right) \qquad \forall k \; .
\end{equation}
\end{definition}

Note that consensus is defined as a Nash equilibrium: each agent minimizes its cost function \emph{assuming the others fixed}; the possibility to decrease cost functions by moving several agents simultaneously is not considered.
Consensus is graph-dependent: agent $k$ reaches consensus when it minimizes its distance to agents $j \rightsquigarrow k$.\vspace{2mm}

\begin{proposition}\label{Prop3}
If $G$ is an equally-weighted complete graph, then the only possible consensus configuration is synchronization.
\end{proposition}\vspace{2mm}

\begin{proof}At consensus the $y_k$ satisfy $\; \; y_k^T \, \sum_{j \neq k} y_j \geq c^T \, \sum_{j \neq k} y_j \quad \forall k \text{ and } \forall c \in \M$.
Furthermore, it is obvious that $y_k^T y_k > y_k^T c$ for any $c \in \M \setminus \lbrace y_k \rbrace$. As a consequence,
$y_k^T \, \sum_{j = 1}^N \, y_j > c^T \, \sum_{j = 1}^N \, y_j \quad \forall c \in \M \setminus \lbrace y_k \rbrace$ and $\forall k$.
Thus according to (\ref{Def:IAM&AIAM:Alt2}), each $y_k$ is located at the $\IAM$ of all the agents, which moreover reduces to a single point; thus $y_k = y_j = \IAM(\lbrace y_l: l=1...N \rbrace)$ $\forall k,j$. \end{proof}\vspace{2mm}

Synchronization is a configuration of complete consensus. To similarly characterize a configuration of complete anti-consensus, it appears meaningful to require that the $\IAM$ of the agents is the entire manifold $\M$; this is called a \emph{balanced} configuration.\vspace{2mm}

\begin{definition}\label{Def5}
$N$ agents are \textbf{balanced} if their $\IAM$ contains all $\M$.
\end{definition}
\vspace{2mm}

Balancing implies some spreading of the agents on the manifold. A full characterization of balanced configurations seems complicated. Balanced configurations do not always exist (typically, when the number of agents is too small) and are mostly not unique (they can appear in qualitatively different forms). The following link exists between anti-consensus for the equally-weighted complete graph and balancing.\vspace{2mm}

\begin{proposition}\label{Prop4}
All balanced configurations are anti-consensus configurations for the equally-weighted complete graph.
\end{proposition}\vspace{2mm}

\begin{proof}
For the equally-weighted complete graph, (\ref{Def:AntiConsensus}) can be written
\begin{equation}\label{Prop4:Pr1}
y_k \in \mathop{\mathrm{argmin}}_{c \in \M} \left(c^T \; (N \, C_e - y_k) \right) \qquad \forall k \; .
\end{equation}
Assume that the agents are balanced. This means that $f(c) = c^T \, C_e$ must be constant over $c \in \M$. Therefore (\ref{Prop4:Pr1}) reduces to $y_k = y_k$ $\forall k$ which is trivially satisfied. \end{proof}
\vspace{2mm}

In contrast to Proposition \ref{Prop3}, Proposition \ref{Prop4} does not establish a necessary and sufficient condition; and indeed, anti-consensus configurations for the equally-weighted complete graph that are not balanced do exist, though they seem exceptional.


\subsection{Examples} The following examples illustrate among others the last assertions about balanced configurations.\vspace{2mm}

\paragraph*{The circle} Anti-consensus configurations for the equally-weighted complete \linebreak graph are fully characterized in \cite{SPL2005}. It is shown that the only anti-consensus configurations that are not balanced correspond to $(N+1)/2$ agents at one position and $(N-1)/2$ agents at the opposite position on the circle, for $N$ odd. Balanced configurations are unique for $N=2$ and $N=3$ and form a continuum for $N > 3$.

Another interesting illustration is the equally-weighted \emph{undirected ring graph} in which each agent is connected to two neighbors such that the graph forms a single closed undirected path. Regular consensus configurations correspond to situations with consecutive agents in the path always separated by the same angle $0 \leq \chi \leq \pi/2$; regular anti-consensus configurations have $\pi/2 \leq \chi \leq \pi$. In addition, for $N \geq 4$, irregular consensus and anti-consensus configurations exist where non-consecutive angles of the regular configurations are replaced by $(\pi - \chi)$. As a consequence:
\begin{enumerate}
\item Several qualitatively different (anti-)consensus configurations exist.
\item Consensus and anti-consensus configurations can be equivalent when discarding the graph. For example, the positions occupied by 7 agents separated by $2 \pi / 7$ (consensus) or $4 \pi / 7$ (anti-consensus) are strictly equivalent; the only difference, based on \emph{which agent} is located at \emph{which position}, concerns the way the links are drawn.
\item Degenerate configurations of simultaneous consensus and anti-consensus exist (e.g. $\chi = \pi / 2$ for $N = 4, 8,...$); this singularity is specific to the particular graph.
\item There is no common anti-consensus state for all ring graphs. Indeed, considering an agent $k$, a common anti-consensus state would require that any two other agents, as potential neighbors of $k$, are either separated by $\pi$ or located at both sides of $k$ at a distance $\chi \geq \pi/2$; one easily verifies that this cannot be satisfied for all $k$.
\end{enumerate}\vspace{2mm}

\paragraph*{The special orthogonal group} Simulations of the algorithms proposed in this paper suggest that balanced configurations always exist for $N \geq 2$ if $n$ is even and for $N \geq 4$ if $n$ is odd. Under these conditions, convergence to an anti-consensus state that is not balanced is not observed for the equally-weighted complete graph.\vspace{2mm}

\paragraph*{The Grassmann manifold} Balanced states on $\Gr$ appear if all eigenvalues of $C_e$ are equal. Since $\mathrm{trace}(C_e) = \frac{1}{N} \sum_k \mathrm{trace}(\Pi_k) = p$, this requires $C_e = \frac{p}{n} I_n$. This is not always possible with $N$ orthonormal projectors of rank $p$. As for $SO(n)$, simulations tend to indicate that it is possible when $N$ is large enough; however, computing the minimal value of $N$ for a given $n$ and $p$ is not straightforward.



\section{Consensus optimization strategy}

The presence of a maximization condition in the definitions of the previous sections naturally points to the use of optimization methods. The present section introduces a cost function whose optimization leads to (anti-)consensus configurations. For a graph $G$ with adjacency matrix $A = [a_{jk}]$ and associated Laplacian $L^{(i)}=[l^{(i)}_{jk}]$ and the variable $y = (y_1,..., y_N) \in \M^N$, define
\begin{equation}\label{Def:P_L}
P_L(y) \;=\; \tfrac{1}{2 N^2} {\textstyle \sum_{k=1}^N \sum_{j=1}^N} \, a_{jk} \; y_j^T y_k \;=\; \xi_1 - \tfrac{1}{4 N^2} {\textstyle \sum_{k=1}^N \sum_{j=1}^N} \, a_{jk} \, \Vert y_j - y_k \Vert^2
\end{equation}
with constant $\xi_1 = \frac{r_{\M}^2}{4 N^2} \sum_k \sum_j a_{jk}$. The index $L$ refers to the fact that (\ref{Def:P_L}) can also be written as a quadratic form on the graph Laplacian:
\begin{equation}\label{Def:P_L:Laplacian}
P_L(y) \;=\;  \xi_2 - \tfrac{1}{2 N^2} {\textstyle \sum_{k=1}^N \sum_{j=1}^N} \, l_{jk}^{(i)} \, y_j^T y_k \qquad \text{with constant } \xi_2 = \tfrac{r_{\M}^2}{2 N^2} {\textstyle \sum_{k} d_k^{(i)}} \, .
\end{equation}
In \cite{MY1} and \cite{TAC2}, this form of $P_L$ is studied on the circle for undirected equally-weighted graphs. For the unit-weighted complete graph, $P := P_L + \frac{r_{\M}^2}{2 N}$ equals
\begin{equation}\label{Def:P:CentroidNorm}
P(y) = \tfrac{1}{2} \Vert C_e \Vert^2 \; ,
\end{equation}
proportional to the squared norm of the centroid $C_e$. This is a classical measure of the synchrony of phase variables on the circle $S^1$, used for decades in the literature on coupled oscillators; in the context of the Kuramoto model, $P(y)$ is known as the ``complex order parameter'' (because $\mathbb{R}^2$ is usually identified with $\mathbb{C}$ in that context). In \cite{SPL2005}, $P$ is used to derive gradient algorithms for synchronization (by maximizing (\ref{Def:P:CentroidNorm})) or balancing (by minimizing (\ref{Def:P:CentroidNorm})) on $S^1$.\vspace{2mm}

\begin{proposition}\label{Prop5}
Synchronization of the $N$ agents on $\M$ is the unique global maximum of $P_L$ whenever the graph $G$ associated to $L^{(i)}$ is weakly connected.
\end{proposition}\vspace{2mm}

\begin{proof}
According to the second form of (\ref{Def:P_L}), $P_L$ reaches its global maximum when $y_j = y_k$ for all $j,k$ for which $a_{jk} \neq 0$. If $G$ is weakly connected, this equality propagates through the whole graph such that $y_1 = y_2 = ... = y_N$.
\end{proof}\vspace{2mm}

\begin{proposition}\label{Prop6} 
Consider $N$ agents on a manifold $\M$ satisfying Assumptions \ref{Ass1} and \ref{Ass2}. Given an undirected graph $G$, a local maximum of the associated cost function $P_L(y)$ necessarily corresponds to a consensus configuration and a local minimum of $P_L(y)$ necessarily corresponds to an anti-consensus configuration.
\end{proposition}\vspace{2mm}

\begin{proof}
The proof is given for local maxima; it is strictly equivalent for local minima. For $y^{\ast} = (y_1^{\ast}... y_N^{\ast})$ to be a local maximizer of $P_L$, $y_k^{\ast}$ must be, for each $k$, a local maximizer of $p_k(c) := P_L(y_1^{\ast}... y_{k-1}^{\ast},\, c,\, y_{k+1}^{\ast}... y_N^{\ast})$. Since $A=A^T$, $p_k$ takes the linear form $p_k(c) = \xi_k + \tfrac{1}{N^2} \, c^T \, ( \sum_j a_{jk} \, y_j^{\ast} )$ with $\xi_k$ constant $\forall k$. Thanks to Assumption \ref{Ass2}, all local maxima of $p_k(c)$ are global maxima. Therefore, $y_k^{\ast}$ is a global maximum of $p_k(c)$ for all $k$, which corresponds to Definition \ref{Def4} of consensus.
\end{proof}\vspace{2mm}

Proposition \ref{Prop6} establishes that a \emph{sufficient} condition for (anti-)consensus configurations is to optimize $P_L$. However, nothing guarantees that this is also \emph{necessary}.  In general, optimizing $P_L$ will thus provide proven (anti-)consensus configurations, but not necessarily all of them (this is because consensus maximizes $P_L$ on $\M^N$ \emph{for only moving one agent with others fixed}, and not along directions of \emph{combined motion of several agents}). The remaining sections of this paper present algorithms that drive the swarm to (anti-)consensus. Being based on the optimization of $P_L$, these algorithms do not necessarily target all possible (anti-)consensus configurations. For instance, for a tree, maximization of $P_L$ always leads to synchronization, although other consensus configurations can exist.


\subsection{Examples} On $SO(n)$ and $\Gr$, $P_L$ with matrix forms for the elements $y_k$ becomes
\begin{equation}\label{Def:P_L:Matrix}
P_L(y) \;=\; \tfrac{1}{2 N^2} {\textstyle \sum_{j=1}^N \sum_{k=1}^N} \, a_{jk} \, \mathrm{trace}(y_j^T y_k) \qquad \text{with } y_k \in \mathbb{R}^{n \times n} \; \forall k \; .
\end{equation}\vspace{2mm}

\paragraph*{The special orthogonal group} Each term $Q_j^T Q_k = Q_j^{-1} Q_k$ is itself an element of $SO(n)$. It is actually the unique element of $SO(n)$ translating $Q_j$ to $Q_k$ by matrix (group) multiplication on the right. Hence, on the Lie group $SO(n)$, the order parameter $P_L$ measures the sum of the traces of the elements translating connected agents to each other. Observing that the trace is maximal for the identity matrix and considering the particular case of $SO(2)$, one can easily imagine how the trace of $Q_j^{-1} Q_k$ characterizes the distance between $Q_j$ and $Q_k$. This cost function has been previously used in \cite{BulloThesis,Sujit1} as a measure of disagreement on $SO(3)$.\vspace{2mm}

\paragraph*{The Grassmann manifold} On $\Gr$, (\ref{Def:P_L:Matrix}) can be rewritten as
$$P_L(\mathcal{Y}) = \tfrac{1}{2 N^2} \, {\textstyle \sum_{j=1}^N \sum_{k=1}^N} \; a_{jk} \, \left( {\textstyle \sum_{i=1}^p} \, \cos^2(\phi_{jk}^i) \right)$$
with $\phi_{jk}^{i} = i^{\mathrm{th}}$ principal angle between $\mathcal{Y}_j$ and $\mathcal{Y}_k$. This reformulation has previously appeared in \cite{Conway, Barg1, PAGrass}.\vspace{2mm}



\section{Gradient consensus algorithms} The previous sections pave the way for ascent and descent algorithms on $P$ and $P_L$. This paper considers continuous-time gradient algorithms, but any descent or ascent algorithm --- in particular, discrete-time --- will achieve the same task; see \cite{PAbook} for extensive information on this subject. In the present paper, the gradient is always defined with the canonical metric induced by the embedding of $\M$ in $\E$.

\subsection{Fixed undirected graphs}
A gradient algorithm for $P_L$ leads to
\begin{equation} \label{Alg:Gradient}
\dot{y}_k(t) =  2 N^2 \alpha \; \mathrm{grad}_{k,\M}(P_L) \; , \qquad k=1...N \; ,
\end{equation}
where $\alpha > 0$ (resp. $\alpha < 0$) for consensus (resp. anti-consensus), $\dot{y}_k$ denotes the time-derivative of agent $k$'s position and $\mathrm{grad}_{k,\M}(f)$ denotes the gradient of $f$ with respect to $y_k$ along $\M$. This gradient can be obtained from the gradient in $\E$,
$$\mathrm{grad}_{k, \E}(P_L) = \tfrac{1}{2 N^2} {\textstyle \sum_j} (a_{jk} + a_{kj}) \; y_j \; ,$$
by orthogonal projection $\mathrm{Proj}_{T\M,k}$ onto the tangent space to $\M$ at $y_k$, yielding $\forall k$
\begin{equation}\label{Alg:Gradient:Proj}
\dot{y}_k(t) = \alpha \; \mathrm{Proj}_{T\M,k} \left( {\textstyle \sum_j} (a_{jk} + a_{kj}) y_j \right) = \alpha \; \mathrm{Proj}_{T\M,k} \left( {\textstyle \sum_j} (a_{jk} + a_{kj}) (y_j-y_k) \right) \, .
\end{equation}
The last equality comes from $\mathrm{Proj}_{T\M,k}(y_k)=0$. It shows that to implement this consensus algorithm, each agent $k$ must know the relative position with respect to itself of all agents $j$ such that $j \rightsquigarrow k$ or $k \rightsquigarrow j$. Since the information flow is restricted to $j \rightsquigarrow k$, (\ref{Alg:Gradient:Proj}) can only be implemented for undirected graphs, for which it becomes
\begin{equation}\label{Alg:Gradient:Undir}
\dot{y}_k(t) = 2 \alpha \; \mathrm{Proj}_{T\M,k} \left( {\textstyle \sum_{j=1}^N} \; a_{jk} (y_j-y_k) \right) \; , \qquad k=1...N \; .
\end{equation}
In the special case of a complete unit-weighted graph,
\begin{equation}\label{Alg:Gradient:Complete}
\vphantom{\underbrace{x}} \dot{y}_k(t) = 2 \alpha N \; \mathrm{Proj}_{T\M,k} \left( C_e(t) - y_k \right) \; , \qquad k=1...N \; .
\end{equation}

\begin{proposition}\label{Prop7}
A group of $N$ agents moving according to (\ref{Alg:Gradient:Undir}) on a manifold $\M$ satisfying Assumptions \ref{Ass1} and \ref{Ass2}, where the graph $G$ associated to $A=[a_{jk}]$ is undirected, always converges to a set of equilibrium points. If $\alpha < 0$, all asymptotically stable equilibria are anti-consensus configurations for $G$. If $\alpha > 0$, all asymptotically stable equilibria are consensus configurations for $G$ (in particular, for the equally-weighted complete graph, the only asymptotically stable configuration is synchronization).
\end{proposition}\vspace{2mm}

\begin{proof}
$\M$ being compact and the $a_{jk}$ bounded, $P_L$ is upper- and lower-bounded. $P_L$ is always increasing (decreasing) for $\alpha > 0$ ($\alpha < 0$) along solutions of (\ref{Alg:Gradient:Undir}), since
$$\dot{P}_L \,=\, {\textstyle \sum_k} \, \dot{y}_k^T \, \mathrm{grad}_{k,\M}(P_L) \,=\, 2 N^2 \alpha \, {\textstyle \sum_k} \, \Vert \mathrm{grad}_{k,\M}(P_L) \Vert^2  \; .$$
By LaSalle's invariance principle, the swarm converges towards a set where $\dot{P}_L = 0$, implying $\mathrm{grad}_{k,\M}(P_L) = 0 \, \Leftrightarrow \, \dot{y}_k = 0$ $\forall k$ and the swarm converges to a set of equilibria. For $\alpha > 0$ ($\alpha < 0$), since $P_L$ always increases (decreases) along solutions, only local maxima (minima) can be asymptotically stable.  Proposition \ref{Prop6} states that all local maxima (minima) of $P_L$ correspond to consensus (anti-consensus).
\end{proof}\vspace{2mm}

\begin{remark}
Computing $\mathrm{grad}_{k,\M}$ directly along the manifold, as in \cite{PAbook}, can be much more efficient if the dimension of $\M$ is substantially lower than $m$ (see \S 6.3).
\end{remark}\vspace{2mm}

\subsection{Extension to directed and time-varying graphs}

Formally, algorithm (\ref{Alg:Gradient:Undir}) can be written for directed and even time-varying graphs, although the gradient property is lost for directed graphs and has no meaning in the time-varying case (since $P_L$ then explicitly depends on time). Nevertheless, the general case of (\ref{Alg:Gradient:Undir}) with varying and directed graphs still exhibits synchronization properties.

It can be shown that synchronization is still a stable equilibrium; it is asymptotically stable if disconnected graph sequences are excluded. Its basin of attraction includes the configurations where all the agents are located in a convex set of $\M$. Indeed, convergence results on Euclidean spaces can be adapted to manifolds when agents are located in a convex set (see e.g. \cite{MOREAU}). On the other hand, examples where algorithm (\ref{Alg:Gradient:Undir}) with $\alpha > 0$ runs into a limit cycle can be built for as simple cases as undirected equally-weighted (but varying) graphs on the circle (see \S 6.3).

Simulations on $SO(n)$ and $\Gr$ seem to indicate that for randomly generated digraph sequences\footnote{More precisely, the following distibution was examined: initially, each element $a_{jk}$ independently takes a value in $\lbrace 0, 1 \rbrace$ according to a probability $\mathrm{Prob}(1)=p$. The corresponding graph remains for a time $t_{graph}$ uniformly distributed in $[t_{min}, t_{max}]$, after which a new graph is built as initially.}, the swarm eventually converges to synchronization when $\alpha > 0$; this would correspond to \emph{generic} convergence for unconstrained graphs.\vspace{2mm}

Algorithm (\ref{Alg:Gradient:Undir}) can lead to a generalization of Vicsek's phase update law (see \cite{VICSEK}) to manifolds. The Vicsek model is a discrete-time algorithm governing the headings of particles in the plane, and hence operates on the circle. It can be written as
\begin{equation}\label{Vicsek}
y_k(t+1) \; \in \; \IAM\left(\lbrace y_j(t) \vert j \rightsquigarrow k \text{ in } G(t) \rbrace \cup \lbrace y_k(t) \rbrace \right)  \; , \qquad k=1...N \; ,
\end{equation}
with the definitions introduced in the present paper; interconnections among particles depend on their relative positions in the plane (so-called ``proximity graphs''). Vicsek's law can be directly generalized in the form (\ref{Vicsek}) to any manifold satisfying Assumption \ref{Ass1}. Based on the previous discussions, it is clear why (\ref{Vicsek}) can be viewed as a discrete-time variant of (\ref{Alg:Gradient:Undir}). When run asynchronously on a fixed undirected graph, (\ref{Vicsek}) is an ascent algorithm for $P_L$; see \cite{MY1} for a precise relationship between the continuous-time and discrete-time consensus algorithms on the circle.

\subsection{Examples} Consensus on the circle is studied in \cite{SPL2005, MY1, MY2, TAC2}; the other algorithms presented here are original.\vspace{2mm}

\paragraph*{The circle}  Denoting angular positions by $\theta_k$, the specific form of (\ref{Alg:Gradient:Undir}) for $S^1$  is
\begin{equation}\label{Alg:Gradient:Circle}
\dot{\theta}_k = \alpha' \; \, {\textstyle \sum_{j=1}^N} \; a_{jk} \, \sin(\theta_k - \theta_j)  \; , \qquad k=1...N \; .
\end{equation}
For the equally-weighted complete graph, this is strictly equivalent to the Kuramoto model \cite{Kuramoto} with identical (zero) natural frequencies.

Algorithm (\ref{Alg:Gradient:Circle}) can run into a limit cycle for varying graphs. Consider a regular consensus state for an equally-weighted ring graph $G_1$, with consecutive agents separated by $\chi < \pi/2$ (local maximum of $P_{L_1}$). Define $G_2$ by connecting each agent to the agents located at an angle $\psi > \pi/2$ from itself with $\psi$ properly fixed. $G_2$ is a collection of disconnected ring graphs and the swarm is at a local minimum of $P_{L_2}$. Starting the system in the neighborhood of that state and regularly switching between $G_1$ and $G_2$, the system will oscillate in its neighborhood, being driven away by $G_2$ and brought back by $G_1$ if consensus is intended and reversely if anti-consensus is intended.\\

\paragraph*{The special orthogonal group} The tangent space to $SO(n)$ at the identity $I_n$ is the space of skew-symetric $n \times n$ matrices. By group multiplication, the projection of $B \in \mathbb{R}^{n \times n}$ onto the tangent space to $SO(n)$ at $Q_k$ is $\, Q_k \, \mathrm{Skew}(Q_k^{-1} B) = Q_k \, (\frac{Q_k^T B}{2} - \frac{B^T Q_k}{2})$. This leads to the following explicit form of algorithm (\ref{Alg:Gradient:Undir}) on $SO(n)$, where the right-hand side only depends on relative positions of the agents with respect to $k$:
\begin{equation}\label{Alg:Gradient:SOn}
Q_k^{-1} \dot{Q}_k = \alpha \; \, {\textstyle \sum_j} \, a_{jk} \, \left( Q_k^{-1} Q_j - Q_j^{-1} Q_k \right)  \; , \qquad k=1...N \; .
\end{equation}
Using Lemma \ref{lem1} in the appendix, the following proves that $SO(n)$ satisfies Assumption \ref{Ass2}. It also includes the proof of Proposition \ref{Prop1}.\vspace{2mm}

\begin{proposition}\label{Prop8}
The manifold $SO(n)$ satisfies Assumption \ref{Ass2}.
\end{proposition}\vspace{2mm}

\begin{proof} (+ Prop.\ref{Prop1}) Consider a linear function $f(Q) = \mathrm{trace}(Q^T B)$ with $Q \in SO(n)$ and $B \in \mathbb{R}^{n \times n}$; $\mathrm{grad}_{\mathbb{R}^{n \times n}}(f) = B$ so $\mathrm{grad}_{SO(n)}(f) =\frac{Q}{2} (Q^T B - B^T Q)$. Since $Q$ is invertible, critical points of $f$ satisfy $(Q^T B - B^T Q) = 0$, meaning that they take the form described by Lemma \ref{lem1}. Using notations of Lemma \ref{lem1}, write $R = H \Lambda H^T$ where $\Lambda$ contains the (non-negative) eigenvalues of $R$. This leads to
$$Q = U H J H^T \quad \Rightarrow \quad Q^T B = H J \Lambda H^T \quad \Rightarrow \quad f(Q) = -{\textstyle \sum_{j=1}^{l}} \; \Lambda_{jj} + {\textstyle \sum_{j=l+1}^{n}} \; \Lambda_{jj} \; .$$
If $l \geq 2$, select any $m \in [2,l]$ and define $Q_{\varepsilon} = U H J A H^T$ where $A$ is the identity matrix except that $A(1,1)=A(m,m)=\cos(\varepsilon)$ and $A(1,m)=-A(m,1)=\sin(\varepsilon)$ with $\varepsilon$ arbitrarily small. It is straightforward to see that $f(Q_{\varepsilon}) > f(Q)$ unless $\Lambda_{11} = \Lambda_{mm} = 0$.
Similarly, if $l = 1$ and $\exists$ $m \geq 2$ such that $\Lambda_{mm} < \Lambda_{11}$, then $f(Q_{\varepsilon}) > f(Q)$ with $Q_{\varepsilon}$ and $A$ defined as previously. Therefore,
\begin{enumerate}
\item if $\det(B) \geq 0$, local maxima require $l=0$ such that $Q=U$ and $f(Q)$ is the sum of the eigenvalues of $R$;
\item if $\det(B) \leq 0$, local maxima require $U$ to take the form of Lemma \ref{lem1} with $l=1$ and $\Lambda_{11} \leq \Lambda_{mm} \forall m$; thus the first column of $H$ corresponds to a smallest eigen- value of $R$ and $f(Q)$ is the sum of $n-1$ largest eigenvalues minus the smallest one.
\end{enumerate}

This shows that all maxima of $f(Q)$ are global maxima (since they all take the same value) and, with $B = C_e$, characterizes the $\IAM$.
\end{proof}\vspace{2mm}

\paragraph*{The Grassmann manifold} The projection of a matrix $M \in \mathbb{S}_n^{+}$ onto the tangent space to $\Gr$ at $\Pi_k$ is given in \cite{SPGrass} as $\Pi_{k} M \Pi_{\bot k} + \Pi_{\bot k} M \Pi_k$. This leads to
\begin{equation}\label{Alg:Gradient:GrassProj}
\dot{\Pi}_k = 2 \alpha \; \, {\textstyle \sum_j} \, a_{jk} \, (\Pi_k \Pi_j \Pi_{\bot k} + \Pi_{\bot k} \Pi_j \Pi_k) \; , \qquad k=1...N \; .
\end{equation}
In practice, the basis representation $Y_k$ is handier than $\Pi_k$ since it involves smaller matrices. Computing the gradient of $P_L(\lbrace \Pi_k \, , \; k=1...N \rbrace) = P_L(\lbrace Y_k \, Y_k^T \, , \; k=1...N \rbrace)$ directly on the quotient manifold as explained in \cite{PAGrass} leads to the algorithm
\begin{equation}\label{Alg:Gradient:GrassStiefel}
\dot{Y}_k = 4 \alpha \; \, {\textstyle \sum_j} \, a_{jk} \; \left( Y_j \, M_{j \cdot k} - Y_k \, M_{j \cdot k}^T M_{j \cdot k} \right)   \; , \qquad k=1...N \; ,
\end{equation}
where the $p \times p$ matrices $M_{j \cdot k}$ are defined as $M_{j \cdot k} = Y_j^T Y_k$. For theoretical purposes, the projector representation is an easier choice, as for the following proofs.\vspace{2mm}

\begin{proposition}\label{Prop9}
The Grassmann manifold satisfies Assumption \ref{Ass2}.
\end{proposition}\vspace{2mm}

\begin{proof} (+ Prop.\ref{Prop2}) Consider a linear function $f(\Pi) = \mathrm{trace}(\Pi^T B)$ where $B \in \mathbb{S}_n^{+}$ and $\Pi$ represents $\mathcal{Y} \in \Gr$; $\mathrm{grad}_{\mathbb{R}^{n \times n}}(f) = B$ so $\mathrm{grad}_{\Gr}(f) = \Pi B \Pi_{\bot} + \Pi_{\bot} B \Pi$. The ranges of the first and second terms in $\mathrm{grad}_{\Gr}(f)$ are at most $\mathcal{Y}$ and its orthogonal complement respectively, so they both equal zero at a critical point $\mathcal{Y}^{\ast}$, such that $\mathcal{Y}^{\ast}$ is an invariant subspace of $B$. In an appropriate basis $(e_1... e_n)$, write $\Pi^{\ast}=\mathrm{diag}(1,... 1, 0,... 0)$ and $B=\mathrm{diag}(\mu_{1},... \mu_{p}, \mu_{p+1}... \mu_{n})$. If $\exists \; d \leq p$ and $l > p$ such that $\mu_d < \mu_l$, then any variation of $\Pi^{\ast}$ rotating $e_d$ towards $e_l$ strictly increases $f(\Pi)$. Therefore, at local maxima of $f(\Pi)$, the $p$-dimensional space corresponding to $\Pi$ must be an eigenspace of $B$ corresponding to $p$ largest eigenvalues. This implies that at any local maximum, $f(\Pi)$ equals the sum of $p$ largest eigenvalues of $B$, so Assumption \ref{Ass2} is satisfied. Replacing $B$ by $C_e$ proves Proposition \ref{Prop2}.
\end{proof}



\section{Consensus algorithms with estimator variables} Section 6 derives algorithms that lead to a consensus situation linked to the interconnection graph. But in many applications, the interconnection graph is just a restriction on communication possibilities, under which one actually wants to achieve a consensus for the complete graph. Moreover, allowing directed and time-varying communication graphs is desirable for robustness. This section presents algorithms achieving the same performance as those of \S 6 for the equally-weighted complete graph --- that is, driving the swarm to synchronization or to a subset of the anti-consensus configurations for the equally-weighted complete graph which seems to contain little more than balancing --- under very weak conditions on the actual communication graph. However, this reduction of information channels must be compensated by adding a consensus variable $x_k \in \E$, which interconnected agents are able to communicate, to the state space of each agent.

\subsection{Synchronization algorithm} For synchronization purposes, the agents run a consensus algorithm on their estimator variables $x_k$ in $\E$, $k=1...N$, initialized arbitrarily but independently and such that they can take any value in an open subset of $\E$; $\forall k$, agent $k$'s position $y_k$ on $\M$ independently tracks (the projection on $\M$ of) $x_k$. This leads to
\begin{eqnarray}
\label{RCA:Synch:General1} \dot{x}_k & \,=\, & \beta \; \, {\textstyle \sum_j} \, a_{jk} \, (x_j - x_k) \phantom{KKKKKKKKKKKKk} , \quad \beta > 0 \\
\label{RCA:Synch:General2} \dot{y}_k & \,=\, & \gamma_S \; \mathrm{grad}_{k,\M}(y_k^T \, x_k) \,=\, \gamma_S \;  \mathrm{Proj}_{T\M,k}(x_k) \phantom{kkkkk} , \quad \gamma_S > 0   \; , \quad k=1...N . \phantom{kkkk}
\end{eqnarray}
Equation (\ref{RCA:Synch:General1}) is a classical consensus algorithm in $\E$, where $\dot{x}_k(t)$ points from $x_k(t)$ towards the centroid of the (appropriately weighted) $x_j(t)$ for which $j \rightsquigarrow k$ at time $t$. According to \cite{MOREAU, MOREAU2, olfati}, if the time-varying communication graph $G(t)$ is piecewise continuous in time and uniformly connected, then all the $x_k$ exponentially converge to a common consensus value $x_{\infty}$; moreover, if $G(t)$ is balanced for all $t$, then $x_{\infty} = \frac{1}{N} \sum_k x_k(0)$ (i.e. $x_{\infty}$ is the centroid of the initial $x_k$). This implies the following convergence property for (\ref{RCA:Synch:General1}),(\ref{RCA:Synch:General2}), where the notation $\IAM_g$ generalizes the definition (\ref{Def:IAM&AIAM:Alt2}) of the $\IAM$ when the points defining $C_e$ are not on $\M$.\vspace{2mm}

\begin{proposition}\label{Prop10}
Consider a piecewise continuous and uniformly connected graph $G(t)$ and a manifold $\M$ satisfying Assumptions \ref{Ass1} and \ref{Ass2}. The only stable limit configuration of the $y_k$ under (\ref{RCA:Synch:General1}),(\ref{RCA:Synch:General2}), with the $x_k$ initialized arbitrarily but independently and such that they can take any value in an open subset of $\E$, is synchronization at $y_{\infty} = \mathrm{Proj}_{T\M,k}(x_{\infty})$; if $G(t)$ is balanced, $y_{\infty} = \IAM_g\lbrace x_k(0),\, k=1...N \rbrace$.
\end{proposition}\vspace{2mm}

\begin{proof}
Convergence of (\ref{RCA:Synch:General1}) towards $x_k = x_{\infty}$ $\forall k$ is proved in \cite{MOREAU2}; the property $x_{\infty} = \frac{1}{N} \sum_k x_k(0)$ for balanced graphs is easy to check (see \cite{olfati}). As a consequence, the asymptotic form of (\ref{RCA:Synch:General1}),(\ref{RCA:Synch:General2}) is a set of $N$ independent systems
\begin{eqnarray}
\label{Proof6_1a} x_k & = & x_{\infty} \\
\label{Proof6_1b} \dot{y}_k & = & \gamma_S \, \mathrm{Proj}_{T\M,k}(x_{\infty})   \; , \qquad k=1...N \; ,
\end{eqnarray}
where $x_{\infty}$ is a constant. According to \cite{ChainRecurrent}, the $\omega$-limit sets of the original system (\ref{RCA:Synch:General1}),(\ref{RCA:Synch:General2}) correspond to the chain recurrent sets of the asymptotic system (\ref{Proof6_1a}),(\ref{Proof6_1b}). The first equation is trivial. According to Proposition 4 in \cite{ChRec2} and Sard's theorem, since (\ref{Proof6_1b}) is a gradient ascent algorithm for $f(y_k) = y_k^T x_{\infty}$ and $f(y_k)$ is smooth (as the restriction of a smooth function to the smooth embedded manifold $\M$), the chain recurrent set of (\ref{Proof6_1b}) is equal to its critical points. Since $x_{\infty}$ is a linear combination of the $x_k(0)$, variations of the $x_k(0)$ are equivalent to variations of $x_{\infty}$.

\emph{Property $o$:} Any open neighborhood $O$ of any point $x_o \in \E$ contains a point $x_a$ for which $f(y_k)$ has a unique (local $=$ global, by Assumption \ref{Ass2}) maximizer.
\newline \emph{Proof:} If $y_k^T \, x_o$ has multiple maximizers, select one of them, call it $y_{\ast}$. Then for $\sigma > 0$,
$y_k^T x_o + \sigma \, y_k^T y_{\ast} \leq y_{\ast}^T x_o + \sigma \, y_k^T y_{\ast} \leq y_{\ast}^T x_o + \sigma \, y_{\ast}^T y_{\ast}$ with equality holding if and only if $y_k = y_{\ast}$, so $y_{\ast}$ is the unique maximizer of $y_k^T \, (x_o + \sigma y_{\ast})$. Since any open neighborhood $O$ of $x_o$ contains points of the form $x_a = x_o + \sigma y_{\ast}$, $\sigma > 0$, property $o$ is proved.

Because of Property $o$, with respect to variations of the $x_k$, the situation ``$f(y_k)$ has multiple maximizers'' is unstable. The situation ``$f(y_k)$ has a unique maximizer'' is stable since it corresponds to a non-empty open set in $\E$; thus a convex neighborhood of $x_{\infty}$ can be found in which the $x_k(t)$ will stay by convexity of (\ref{RCA:Synch:General1}) and where $f(y_k)$ has a unique maximizer. With respect to variations of the $y_k$, the (thus unique) maximizer of $y_k^T \, x_k$ is the only stable equilibrium for gradient ascent algorithm (\ref{RCA:Synch:General2}), such that for $x_k \rightarrow x_{\infty}$ the only stable situation is synchronization.
\end{proof}

\subsection{Anti-consensus algorithm} For anti-consensus, in analogy with the previous section, each $y_k$ evolves according to a gradient algorithm to maximize its distance to $x_k(t)$. If $x_k(t)$ asymptotically converges to $C_e(t)$, this becomes equivalent to the gradient anti-consensus algorithm (\ref{Alg:Gradient:Complete}). Imposing $x_k(0) = y_k(0)$ $\forall k$, the following algorithm achieves this purpose when $G(t)$ is balanced $\forall t$:
\begin{eqnarray}
\label{RedBal:Gen1} \dot{x}_k & \,=\, & \beta \; \, {\textstyle \sum_{j=1}^N} \; a_{jk} \, (x_j - x_k) \; + \dot{y}_k \phantom{KKKKKKKjkj} , \quad \beta > 0\\
\label{RedBal:Gen2} \dot{y}_k & \,=\, & \gamma_B \, \mathrm{grad}_{k,\M}(y_k^T x_k) \,=\, \gamma_B \,  \mathrm{Proj}_{T\M,k}(x_k) \phantom{KKK} , \quad \gamma_B < 0   \; , \quad k=1...N \; . \phantom{kkkk}
\end{eqnarray}
Note that the variables $x_k$ and $y_k$ are fully coupled; in a discrete-time version of this system, this essential feature of the algorithm must be retained in the form of \emph{implicit} update equations in order to ensure convergence (see \cite{MY2} for details).\vspace{2mm}

\begin{proposition}\label{Prop11}
Consider a piecewise continuous, uniformly connected and balanced graph $G(t)$ and a manifold $\M$ satisfying Assumptions \ref{Ass1} and \ref{Ass2}. Then, algorithm (\ref{RedBal:Gen1}),(\ref{RedBal:Gen2}) with initial conditions $x_k(0) = y_k(0)$ $\forall k$ converges to an equilibrium configuration of the anti-consensus algorithm for the equally-weighted complete graph, that is (\ref{Alg:Gradient:Complete}) with $\alpha < 0$.
\end{proposition}\vspace{2mm}

\begin{proof}
First show that $\tfrac{1}{N} \sum_k x_k(t) = \tfrac{1}{N}\sum_k y_k(t) = C_e(t)$. Since $x_k(0) = y_k(0)$ $\forall k$, it is true for $t=0$. Thus it remains to show that $\sum_k \dot{x}_k(t) = \sum_k \dot{y}_k(t)$. This is the case because a balanced graph ensures that the first two terms on the right side of the following expression cancel each other:
$${\textstyle \sum_k} \, \dot{x}_k(t) \; = \; \beta \; {\textstyle \sum_j} \, \left( {\textstyle \sum_k} \, a_{jk} \right) \, x_j \;-\; \beta \; {\textstyle \sum_k} \, \left( {\textstyle \sum_j} \, a_{jk} \right) \, x_k \;+\; {\textstyle \sum_k} \, \dot{y}_k(t) \; .$$

Next, prove that $\forall k$, $\dot{y}_k(t)$ is a uniformly continuous function in $L_2(0,+\infty)$ such that Barbalat's Lemma implies $\dot{y}_k \rightarrow 0$. First show that $W(t) = \frac{1}{2} \sum_k x_k(t)^T x_k(t)$ is never increasing along the solutions of (\ref{RedBal:Gen1}),(\ref{RedBal:Gen2}). Denoting by $(x)_{j}$, $j=1...m$, the vectors of length $N$ containing the $j$-th component of every $x_k$, $k=1...N$ and by $L^{(i)}$ the in-Laplacian of the varying graph associated to the $a_{jk}$, one obtains
$$\dot{W}(t) \;=\; {\textstyle \sum_k} \, x_k^T \dot{x}_k \;=\; {\textstyle \sum_k} \, x_k^T \dot{y}_k \;-\; \beta \; {\textstyle \sum_j} (x)_{j}^T L^{(i)} (x)_{j} \; .$$
The term containing $\, L^{(i)} \,$ is non-positive because the Laplacian of balanced graphs is positive semi-definite (see \cite{Wi}). Replacing $\dot{y}_k$ from (\ref{RedBal:Gen2}) and noting that \linebreak
$x_k^T \; \mathrm{Proj}_{T\M,k}(x_k) \;=\; \left( \mathrm{Proj}_{T\M,k}(x_k) \right)^T \; \mathrm{Proj}_{T\M,k}(x_k)$, one obtains
\begin{equation}\label{ZZZ}
\dot{W}(t) \;=\; \gamma_B \; {\textstyle \sum_k} \, \Vert \mathrm{Proj}_{T\M,k}(x_k) \Vert^2 \;-\; \beta \; {\textstyle \sum_j} \, (x)_{j}^T L^{(i)} (x)_{j} \leq 0 \; .
\end{equation}
Thus $W(t) \leq W(0) = \frac{N}{2} \, r_{\M}^2$ which implies that each $\dot{y}_k(t)$ is in $L_2(0,+\infty)$ since
$$\tfrac{1}{\vert \gamma_B \vert} \; {\textstyle \sum_k \int_0^{+\infty}} \, \Vert \dot{y}_k(t) \Vert^2 \, dt \;\leq\; {\textstyle -\int_0^{+\infty}} \, \dot{W}(t) \, dt \; \leq \; \tfrac{N}{2} \, r_{\M}^2 \; .$$
$W(t) \leq W(0)$ also implies that $x_k$ is uniformly bounded $\forall k$; from (\ref{RedBal:Gen2}), $\dot{y}_k$ is uniformly bounded as well. Combining these two observations, with the $a_{jk}$ bounded, (\ref{RedBal:Gen1}) shows that $x_k$ has a bounded derivative and hence is Lipschitz in $t$ $\forall k$. Now write
$$\Vert \dot{y}_k(x_k(t_1),y_k(t_1)) - \dot{y}_k(x_k(t_2),y_k(t_2)) \Vert \; \leq \phantom{kkkkkkkkkkkkkkkkkkkkkkkkkkkkkkkkkk}$$
\vspace{-12mm}

$$\phantom{kkk} \Vert \dot{y}_k(x_k(t_1),y_k(t_1)) - \dot{y}_k(x_k(t_2),y_k(t_1)) \Vert + \Vert \dot{y}_k(x_k(t_2),y_k(t_1)) - \dot{y}_k(x_k(t_2),y_k(t_2)) \Vert \, .$$
The first term on the second line is bounded by $r_1 \, \vert t_1 - t_2 \vert$ for some $r_1$ since $\dot{y}_k$ is linear in $x_k$ and $x_k$ is Lipschitz in $t$. The second term on the second line is bounded by $r_2 \, \vert t_1 - t_2 \vert$ for some $r_2$ since $\dot{y}_k$ is Lipschitz in $y_k$ (as the gradient of a smooth function along the smooth manifold $\M$) and $\tfrac{d}{dt}(y_k) = \dot{y}_k$ is uniformly bounded. Hence, $\dot{y}_k$ is Lipschitz in $t$ and therefore uniformly continuous in $t$, such that Barbalat's Lemma can be applied. Therefore $\dot{y}_k \rightarrow 0$. Thus from 
\cite{ChainRecurrent}, the $\omega$-limit sets of (\ref{RCA:Synch:General1}),(\ref{RCA:Synch:General2}) correspond to the chain recurrent sets of the asymptotic system
\begin{eqnarray*}
\dot{x}_k & \,=\, & \beta \; \, {\textstyle \sum_j} \, a_{jk} \, (x_j - x_k) \\
0 & = & \gamma_B \,  \mathrm{Proj}_{T\M,k}(x_k) \; .
\end{eqnarray*}
The second line is just a static condition. The chain recurrent set of the linear consensus algorithm in the first line reduces to its equilibrium set $x_k = x_{\infty}$ $\forall k$. But then, from the beginning of the proof, $x_k= C_e$ $\forall k$ such that the static condition becomes $0 = \gamma_B \,  \mathrm{Proj}_{T\M,k}(C_e)$ $\forall k$. This is the condition for an equilibrium of anti-consensus algorithm (\ref{Alg:Gradient:Complete}) with $\gamma_B = 2 \alpha N$.
\end{proof}

In simulations, a swarm applying (\ref{RedBal:Gen1}),(\ref{RedBal:Gen2}) with $x_k(0) = y_k(0)$ $\forall k$ seems to generically converge to an anti-consensus configuration of the equally-weighted complete graph, that is a \emph{stable} equilibrium configuration of (\ref{Alg:Gradient:Complete}) with $\alpha < 0$.

\subsection{Examples} Applying this strategy to the circle yields the results of \cite{MY2}, the $x_k$ reduce to vectors of $\mathbb{R}^2$; algorithms (\ref{RCA:Synch:General2}) and (\ref{RedBal:Gen2}) respectively drive the $y_k$ towards and away from the central projection of $x_k$ onto the unit circle.\vspace{2mm}

\paragraph*{The special orthogonal and Grassmann manifolds} The particular balancing algorithms will not be detailed as they are directly obtained from their synchronization counterparts. Introducing auxiliary $n \times n$-matrices $X_k$, (\ref{RCA:Synch:General1}) may be transcribed verbatim. Using previously presented expressions for $\mathrm{Proj}_{T\M,k}(X_k)$, (\ref{RCA:Synch:General2}) becomes
\begin{eqnarray}
\label{RCA:Synch:SOn2} \text{On } SO(n): \quad \phantom{ikKKK} Q_k^{-1}\dot{Q}_k & \;=\; & \tfrac{\gamma_S}{2} \, \left( Q_k^T X_k - X_k^T Q_k \right) \phantom{kkkkkk} , \quad k=1...N \, . \phantom{kkk} \\[2mm]
\label{RCA:Synch:Gr2} \text{On } \Gr: \quad \phantom{kKKK} \dot{\Pi}_k & \;=\; & \gamma_S \, \left( \Pi_k X_k \Pi_{\bot k} + \Pi_{\bot k} X_k \Pi_k \right)   \; , \quad k=1...N \, . \phantom{kkk}
\end{eqnarray}
Note that for $\Gr$, the projector representation must be used in (\ref{RCA:Synch:General1}) and (\ref{RedBal:Gen1}), such that using $n \times n$ matrices $X_k$ becomes unavoidable.

\subsection{Remark about the communication of estimator variables} To implement the algorithms of this section, interconnected agents must communicate the values of their estimator variable $x_k$. It is important to note that the variables $x_k$ may not just be a set of abstract scalars for each agent $k$: since $x_k$ interacts with the geometric $y_k$, it must be a geometric quantity too. However, the $x_k$ evolve in $\E$ while the original system lives on $\M$; the relative position of agents on $\M$ is a meaningful measurement, but nothing ensures \emph{a priori} that a similar thing can be done in $\E$. A solution could be to use a common (thus external) reference frame in $\E$ and transmit the coordinates of the $x_k$ in this frame. That solution would unfortunately imply that the swarm loses its full autonomy; however, the external frame is just used for ``translation'' purposes and does not interfer with the dynamics of the system.

When $\M$ is (a subgroup of) $SO(n)$, the algorithms can be reformulated such that they work completely autonomously if interconnected agents measure their relative positions $Q_k^T Q_j$. Indeed, define $Z_k = Q_k^T X_k$. Then (\ref{RCA:Synch:General1}),(\ref{RCA:Synch:General2}) for instance becomes
\begin{eqnarray}
\label{RCA:Shape:S1} \dot{Z}_k & \;=\; & (Q_k^T \dot{Q}_k)^T Z_k \;+\; \beta \; {\textstyle \sum_j} \, a_{jk} \left( (Q_k^T Q_j) Z_j - Z_k \right)\\
\label{RCA:Shape:S2} Q_k^T \dot{Q}_k & \;=\; & \tfrac{\gamma_S}{2} \; \left( Z_k - Z_k^T \right) \phantom{kkkkkkkkkkkkkkkkkkkkkkkkk}, \; k=1...N \; . \phantom{kkk}
\end{eqnarray}
In this formulation, each agent $k$ can represent $Z_k$ as an array of scalars, whose columns express the column-vectors of $X_k$ as coordinates in a local frame attached to $k$ (i.e. in a frame rotated by $Q_k$ with respect to a hypothetical reference frame). Pre-multiplying $Z_j$ by $Q_k^T Q_j$ expresses $X_j$ in the local frame of $k$, and $Q_k^T \dot{Q}_k$ expresses the velocity of $Q_k$ (with respect to a hypothetical fixed reference) in the local frame of $k$ as well. Thus (\ref{RCA:Shape:S1}),(\ref{RCA:Shape:S2}) actually corresponds to (\ref{RCA:Synch:General1}),(\ref{RCA:Synch:General2}) written in the local frame of $k$. Each agent $k$ gets from its neighbors $j \rightsquigarrow k$ their relative positions $Q_k^T Q_j$ and the $n \times n$ arrays of numbers $Z_j$; from this it computes the update $\dot{Z}_k$ to its own array of numbers $Z_k$ and the move it has to make with respect to its current position, $Q_k^T \dot{Q}_k$. The same can be done for the anti-consensus algorithm.\vspace{2mm}


\section{Conclusion}

The present paper makes three main contributions.

First, it defines the induced arithmetic mean of $N$ points on an embedded connected compact homogeneous manifold $\M$; though it differs from the traditional Karcher mean, it has a clear geometric meaning with the advantage of being easily computable --- see analytical solutions for $SO(n)$ and $\Gr$.

Secondly, a definition of consensus directly linked to the induced arithmetic mean is presented for these manifolds. In particular, the notion of balancing introduced in \cite{SPL2005} for the circle is extended to connected compact homogeneous manifolds. Consensus for the equally-weighted complete graph is equivalent to synchronization. Likewise, it appears in simulations that anti-consensus for the equally-weighted complete graph leads to balancing (if $N$ is large enough), even though this could not be proved.

Thirdly, consensus is formulated as an optimization problem and distributed consensus algorithms are designed for $N$ agents moving on a connected compact homogeneous manifold. In a first step, gradient algorithms are derived for fixed undirected interconnection graphs; (anti-)consensus configurations are their only stable equilibria. Similar algorithms are considered when the graph is allowed to be directed and/or to vary, but their convergence properties are mostly open. In a second step, the algorithms are modified by incorporating an estimator variable for each agent. In this setting, convergence to the (anti-)consensus states of the equally-weighted complete graph can be established theoretically for time-varying and directed interconnection graphs. The meaningful way of communicating estimators between agents remains an open issue when $\M$ is not a subgroup of $SO(n)$.

Running examples $SO(n)$ and $\Gr$ illustrate the validity of the discussion and provide geometric insight. The models and results obtained by applying this framework to the circle are strictly equivalent to existing models and results (most significantly in \cite{SPL2005},\cite{TAC2},\cite{MY2}). This draws a link from the present discussion to the vast literature about synchronization and balancing on the circle.\vspace{2mm}


\section{Appendix}
\begin{lemma}\label{lem1} If $g(Q) = Q^T B - B^T Q$ with $Q \in SO(n)$ and $B \in \mathbb{R}^{n \times n}$, then $g(Q) = 0$ iff $Q = U H J H^T$, where $B = U R$ is a polar decomposition of $B$, the columns of $H$ contain (orthonormalized) eigenvectors of $R$ and 
$$J=\left(
\begin{array}{cc}
-I_{l} & 0 \\
0 & I_{n-l}
\end{array}
\right) \; , \qquad \begin{array}{ll}
l \text{ even} & \text{if } \det(U) > 0 \\ l \text{ odd} & \text{if } \det(U) < 0
\end{array}
$$
\end{lemma}

\begin{proof}
All matrices $Q$ of the given form obviously satisfy that $Q^T B$ is symmetric. The following constructive proof shows that this is the only possible form.\vspace{2mm}

Since $U^T B = R$ is symmetric with $U \in O(n)$, the problem is to find all matrices $T = U^T Q \in O(n)$ such that $S=T^T R$ is symmetric and $\det(T) = \det(U)$. Work in a basis of eigenvectors $H^{\ast}$ diagonalizing $R$ with its eigenvalues placed in decreasing order $\lambda_1 \geq \lambda_2 ... \geq \lambda_n \geq 0$. The following shows that $T$ is diagonal in that basis. Then orthogonality of $T$ imposes values $1$ or $-1$ on the diagonal, the number $l$ of $-1$ being compatible with $\det(T) = \det(U)$; the final form follows by returning to the original basis and reordering the eigenvectors such that those corresponding to $-1$ are in the first columns.

The $j^{th}$ column of $S$ is simply the $j^{th}$ column of $T$ multiplied by $\lambda_j$. Therefore:
\begin{enumerate}
\item If $\lambda_i = \lambda_j$, then $H^{\ast}$ may be chosen such that the corresponding submatrix $T(i:j,\, i:j) = $ intersection of rows $i$ to $j$ and columns $i$ to $j$ of $T\; \;$ is diagonal.
\item If $\lambda_{p+1} = 0$ and $\lambda_p \neq 0$, then $S$ symmetric implies $T(n-p:n,\, 1:p) = 0$. As $T(n-p:n,\, n-p:n)$ is diagonal from 1., only diagonal elements are non-zero in the last $n-p$ rows of $T$. Rows and columns of $T$ being normalized, $T(1:p,\, n-p:n) = 0$.
\item Consider $i_- \leq p$ and $i_+$ the smallest index such that $\lambda_{i_+} < \lambda_{i_-}$. Note that
\begin{equation}\label{Proof:Lemma}
{\textstyle \sum_j T_{i_-j}^2 = \sum_j T_{ji_-}^2 = 1} \text{ (orthogonality) and } \; {\textstyle \sum_j S_{i_-j}^2 = \sum_j S_{ji_-}^2} \text{ (symmetry).}
\end{equation}
\end{enumerate}
Start with $i_- = 1$ and assume $\lambda_{i_+} > 0$. (\ref{Proof:Lemma}) can only be satisfied if $T_{jk} = T_{kj} = 0$ $\forall j \geq i_+$ and $\forall k \in [i_, i_+)$; 1.~further implies $T_{jk} = T_{kj} = 0$ $\forall j \neq k$ and $\forall k \in [i_-, i_+)$. This argument is repeated by defining the new $i_{-}$ as being the previous $i_{+}$ until $\lambda_{i_+} = 0$ (case 2.) or $\lambda_{i_{-}} = \lambda_n > 0$. This leaves $T$ diagonal.
\end{proof}
\vspace{2mm}



\end{document}